\documentclass[a4paper]{article}
\usepackage{sectsty}
\usepackage{amsmath}
\usepackage{amsthm}
\usepackage{dsfont}
\usepackage{pxfonts}

\newtheorem{teo}{Theorem}
\newtheorem{prop}{Proposition}
\newtheorem{coro}{Corollary}
\newtheorem{lem}{Lemma}
\newtheorem{note}{Note}

\makeatletter
\def\@seccntformat#1{\@ifundefined{#1@cntformat}%
{\csname the#1\endcsname\quad}% default
{\csname #1@cntformat\endcsname}% individual control
}
\def\section@cntformat{\thesection. \quad}
\def\subsection@cntformat{\thesubsection . \quad}
\makeatother

\begin{document}

\def \ni {\noindent}
\allsectionsfont{\mdseries\itshape\centering}

\title{A new geometric description for Igusa's modular form $(azy)_5$}
\author{A. Fiorentino}
\date{}
\maketitle

\begin{abstract}
\ni The modular form $(azy)_5$ notably appears in one of Igusa's classic structure theorems as a generator of the ring of full modular forms in genus $2$, being exhibited by means of a complicated algebraic expression. In this work a different description for this modular form is provided by resorting to a peculiar geometrical approach.
\end{abstract}

\medskip

\medskip

\section{Definitions and Notations}
The symplectic group $Sp(2g,\mathds{R})$ acts biholomorphically and transitively on the \emph{Siegel upper half-plane} $\mathfrak{S}_g$, namely the tube domain of complex symmetric $g \times g$ matrices with positive definite imaginary part; its action is defined by:\\

\begin{equation}
\label{act}
\begin{array}{ccc}
\gamma \cdot \tau \coloneqq (a \tau + b) \cdot (c \tau + d)^{-1} & & \gamma = \begin{pmatrix} a & b \\ c & d \end{pmatrix} \in Sp(2g,\mathds{R})
\end{array}
\end{equation}\\

\ni Furthermore, for each $k \in \mathds{Z}$ the non-vanishing function:
\begin{equation*}
 D(\gamma , \tau)^k \coloneqq {\det{(c \tau + d)}}^k \quad \quad \forall \, (\gamma, \tau) \in Sp(2g,\mathds{R}) \times \mathfrak{S}_g
\end{equation*} is a factor of automorphy, \emph{i.e.} a holomorphic function on $\mathfrak{S}_g$ satisfying the cocycle condition:
\begin{equation}
\label{cocycle}
D(\gamma \gamma' , \tau)^k = D(\gamma , \gamma' \tau)^k \, D(\gamma', \tau)^k \quad \, \forall \gamma , \gamma' \in Sp(2g, \mathds{R})
\end{equation} Thanks to this property, an action of $Sp(2g,\mathds{R})$ is also well defined on the space of holomorphic functions on $\mathfrak{S}_g$ for each $k \in \mathds{Z}$ by means of the rule:\\

\begin{equation}
\label{azf1}
(\gamma|_kf) (\tau) \coloneqq D(\gamma^{-1}, \tau)^{-k} \, f(\gamma^{-1} \tau)
\end{equation}\\

The action of the so-called {\it Siegel modular group} $\Gamma_g \coloneqq Sp(2g, \mathds{Z})$ is particularly central to the theory of modular forms; a \emph{Siegel modular form}, or simply a \emph{modular form}, of weight $k \in \mathds{Z}^+$ with respect to a subgroup $\Gamma$ of finite index in $\Gamma_g$ is a holomorphic function $f: \mathfrak{S}_g \rightarrow \mathds{C}$ satisfying the property \footnote{When $g=1$ a modular form is also demanded to be holomorphic on the cusp $\infty$.}:
\begin{equation}
\label{modularity}
\gamma^{-1}|_kf = f \quad \quad \forall  \gamma \in \Gamma
\end{equation}

\ni Modular forms with respect to a subgroup $\Gamma$ of finite index in $\Gamma_g$ form a ring $A(\Gamma)$ which is positively graded by the weights. \\

\ni More generally, whenever $\chi$ is a character of $\Gamma$, a modular form with weight $k \in \mathds{Z}^+$ and character $\chi$ is a holomorphic function $f: \mathfrak{S}_g \rightarrow \mathds{C}$ satisfying:
\begin{equation*}
\label{modularity}
\gamma^{-1}|_kf = \chi(\gamma) \, f \quad \quad \forall  \gamma \in \Gamma
\end{equation*} \\

\ni As for subgroups of finite index in $\Gamma_g$, they are, in fact, characterized by containing for some $n \in \mathds{N}$ the so-called \emph{principal congruence subgroup of level $N$}:
\begin{equation*}
\Gamma_g(n) = \{ \gamma \in \Gamma_g \mid \gamma \equiv 1_{2g} \, \text{mod} \, n  \}
\end{equation*} which is of finite index itself and normal in $\Gamma_g$; remarkable families of such subgroups are:
\begin{equation*}
\begin{array}{l}
\Gamma_g(n,2n) \coloneqq \{ \, \gamma \in \Gamma_g (n) \mid \, \, \, diag(a{}^t b) \equiv  diag(c{}^t d) \equiv 0 \, \text{mod} \, 2n \, \}\\
\\
\Gamma_{g,0}(n) \coloneqq \{ \, \gamma \in \Gamma_g \mid \, \, \, c \equiv 0 \, \text{mod} \,  n \, \}
\end{array}
\end{equation*} \\

\medskip

An outstanding role in constructing modular forms is actually played by {\it Riemann Theta functions with characteristics}; for each $m=(m' , m'') \in \mathds{Z}^g \times \mathds{Z}^g$ they are defined on $\mathfrak{S}_g \times \mathds{C}^g$ by the series:
\begin{equation*}
\theta_{m}(\tau , z) \coloneqq \sum_{n \in \mathds{Z}^g} \text{exp} \left \{ {}^t \left ( n + \frac{m'}{2} \right ) \tau \left ( n + \frac{m'}{2} \right ) + 2 {}^t \left ( n + \frac{m'}{2} \right ) \left ( z + \frac{m''}{2} \right ) \right \}
\end{equation*} where $\text{exp}(w)$ stands for $e^{\pi i w}$. Since $\theta_{m+2n}(\tau , z) = (-1)^{{}^tm'n''} \theta_{m} (\tau , z )$, these functions are parametrized merely by means of reduced {\it $g$-characteristics}, namely vector columns $\begin{bmatrix} m' \\ m'' \end{bmatrix}$ with $m', m'' \in \mathds{Z}_{2}^g$.\\
\ni Throughout this paper reduced $g$-characteristics will be simply referred to as $g$-characteristics, their set being henceforward conventionally denoted by the symbol $C^{(g)}$. For each $m \in C^{(g)}$ the {\it Theta constant} $\theta_m$ with $g$-characteristic $m$ is defined by setting:
\begin{equation*}
\theta_m (\tau) \coloneqq \theta_m(\tau , 0)
\end{equation*}
\ni The only non-vanishing Theta constants are plainly seen to be those related to even characteristics, namely characteristics $m$ satisfying $(-1)^{^t m' m''}=1$. More precisely, a parity function is defined on $g$-characteristics by:
\begin{equation*}
e(m) \coloneqq (-1)^{^t m' m''} \quad \quad \forall \, m=\begin{bmatrix} m' \\ m'' \end{bmatrix} \in C^{(g)}
\end{equation*} thus dividing them into {\it even} and {\it odd} respectively if $e(m)=1$ or $e(m)=-1$;  there are then $2^{g-1}(2^g+1)$ distinct non-vanishing Theta constants for each $g \geq 1$, each being related to an even $g$-characteristic, while the remaining $2^{g-1}(2^g-1)$ Theta constants that are associated to the odd $g$-characteristics are trivial functions. Henceforward, the symbols $C_e^{(g)}$ and $C_o^{(g)}$ will stand respectively for the set of even $g$-characteristics and the set of odd $g$-characteristics.\\

\ni An action of the group $\Gamma_g$ is well defined on $C^{(g)}$ by:
\begin{equation}
\label{action}
\gamma \begin{bmatrix} m' \\ m'' \end{bmatrix} \coloneqq \left [ \begin{pmatrix} d &  -c \\ - b &  a\end{pmatrix} \begin{pmatrix} m' \\ m'' \end{pmatrix} + \begin{pmatrix} diag(c{}^t  d ) \\ diag( a{}^t  b)  \end{pmatrix} \right ] mod \, 2 
\end{equation} $\Gamma_g$ acts both on $C_e^{(g)}$ and on $C_o^{(g)}$, the parity being preserved, and the action of the subgroup $\Gamma_g(2)$ is, in particular, trivial. The action of $\Gamma_g$ on $k$-plets of even $2$-characteristics is also worth being briefly highlighted; when $g=2$, a group isomomorphism between $\Gamma_2/\Gamma_2(2)$ and the symmetric group $S_6$ is naturally defined:
\begin{equation}
\label{isomorphismtoS6}
\psi_P : \Gamma_2 / \Gamma_2(2) \mapsto S_6
\end{equation}
by merely focusing on the action of $\Gamma_2 / \Gamma_2(2)$ on the six elements of $C_o^{(2)}$. By pointing to the action of $S_6$ on non ordered $k$-plets of even $2$-characteristics, the set of non-ordered $3$-plets of even $2$-characteristics, in particular, is found to decompose into two orbits (cf. \cite{Structure}):\\
\begin{equation*}
\begin{array}{l}
 C_{3}^{-}= \{ \, \{m_1, m_2, m_3\} \, \, \, | \, \, \, e(m_1 + m_2 + m_3)=-1 \, \} \\
\\
 C_{3}^{+}= \{ \, \{m_1, m_2, m_3\} \, \, \, | \, \, \, e(m_1 + m_2 + m_3)=1  \, \}
\end{array}
\end{equation*} \\

\ni with $card(C_{3}^{-}) = card(C_{3}^{+}) = 60$, while non-ordered $4$-plets of even $2$-characteristics decompose into three orbits $C_4^+$, $C_4^{\,*}$ and $C_4^-$ (cf. \cite{Structure}), where, in particular:\\

\begin{equation*}
\begin{array}{l}
 C_{4}^{-}= \{ \, \{m_1, \dots, m_4\} \, \, \, | \, \, \{m_i, m_j, m_k\} \in C_{3}^{-}\, \, \, \forall   \{m_i, m_j, m_k\} \subset  \{m_1, \dots, m_4\} \, \} \\
\\
 C_{4}^{+}= \{ \, \{m_1, \dots , m_4\} \, \, \, | \, \, \{m_i, m_j, m_k\} \in C_{3}^{+} \, \, \, \forall   \{m_i, m_j, m_k\} \subset  \{m_1, \dots, m_4\} \, \}
\end{array}
\end{equation*} \\

\ni with  $card(C_{4}^{-}) = card(C_{4}^{+}) = 15$.\\

\ni With reference to the actions described in (\ref{act}) and (\ref{action}), the transformation law for Riemann Theta function can be now outlined; for each $m \in C^{(g)}$ and $\gamma \in \Gamma_g$ one has (cf. \cite{Igusa10} or \cite{Igusa2}):

\begin{equation*}
\label{generaltranformtheta}
\begin{array}{l}
\theta_{\gamma m} (\gamma \tau \, , {}^t(c \tau + d)^{-1} z) = \kappa (\gamma) \, \chi_m(\gamma) \, \phi(\gamma, \tau, z) \, {\det{(c \tau + d)}}^{\frac{1}{2}} \theta_m (\tau, z) \\
\\
\forall \tau \in \mathfrak{S}_g \quad \forall z \in \mathds{C}^g
\end{array}
\end{equation*}\\

\ni where:

\begin{enumerate}
\item $\kappa$ is such a function that $\kappa (\gamma )^4 = exp \{ Tr({}^tbc) \}$ whenever $\gamma \in \Gamma_g$ and, more particularly, $\kappa (\gamma)^2 = exp \left \{ \frac{1}{2} Tr(a-1_g) \right \}$ whenever $\gamma \in \Gamma_2(2)$;
\item $\chi_m (\gamma) = \text{exp} \{ 2 \xi_m(\gamma) \}$ with: 
\begin{equation*}
\begin{split}
\xi_m(\gamma)  = & -\frac{1}{8} ({}^t m ' {}^tb d m' + {}^t m '' {}^t ac m'' - 2{}^t m ' {}^t b c m'') + \\
                                         &  -\frac{1}{4} {}^tdiag(a{}^t b)(dm ' - c m '')
\end{split}
\end{equation*}
\item $\phi (\gamma, \tau, z) = \text{exp} \left \{ \frac{1}{2} {}^t z \,  [(c \tau + d)^{-1} c] \, z \right \}$;
\item The branch of $\det{(c \tau + d)}^{\frac{1}{2}}$ is the one whose sign is positive whenever $Re\tau = 0$;
\end{enumerate} Since $\phi |_{z=0} = 1$, Theta constants transform as follows:
\begin{equation*}
\theta_{\gamma m} (\gamma \tau) = \kappa(\gamma) \, \chi_m(\gamma) \, {\det{(c \tau + d)}}^{\frac{1}{2}} \theta_m (\tau)
\end{equation*}
\ni the product $\theta_m \theta_n$ of two Theta constants being thus found to be a modular form of weight $1$ with respect to $\Gamma_g(4,8)$.\\

\ni \emph{Second order Theta constants} can be also defined by setting for each $m' \in \mathds{Z}_2^g$:
\begin{equation*}
\Theta_{m'}(\tau) \coloneqq \theta_{\left [\substack{m' \\ 0} \right ]}(2\tau, 0)=\theta_{\left [\substack{m' \\ 0} \right ]}(2\tau)
\end{equation*} Whenever $\gamma \in \Gamma_{g,0}(2)$, a remarkable transformation law holds for second order Theta constants:
\begin{equation*}
\Theta_{m'}(\gamma \tau) = \theta_{\left [\substack{m' \\ 0} \right ]}(\tilde{\gamma} 2\tau) = \kappa(\tilde{\gamma}) \chi_{\left [\substack{m' \\ 0} \right ]}(\tilde{\gamma}) \det{(c\tau + d)}^{\frac{1}{2}}\Theta_{m'}(\tau)
\end{equation*}
\ni where $\tilde{\gamma} \coloneqq \begin{pmatrix} a & 2b \\ \frac{1}{2}c & d \end{pmatrix} \in \Gamma_g$.\\

\ni By virtue of this formula the product $\Theta_{m'}\Theta_{n'}$ of two second order Theta constants is likewise found to be a modular form of weight $1$ with respect to $\Gamma_g(2,4)$. \\

\ni The product of all the non-vanishing Theta constants is a modular form of weight $2^{g-2}(2^g+1)$ with respect to $\Gamma_g$ itself whenever $g\geq 3$. The case $g=2$ is instead a special one, the character appearing in the corresponding transformation formula being not trivial; in this case a modular form of weight $5$ only with respect to $\Gamma_2(2)$ is actually gained:
\begin{equation}
\label{chi5}
\chi_5 \coloneqq \prod_{m \in C_{e}^{(2)}} \theta_m = \mu \cdot \det \begin{pmatrix} \, \Theta_{[ 0 \, 0]} &  \Theta_{[ 0 \, 1]} &  \Theta_{[ 1 \, 0]} &  \Theta_{[ 1 \, 1]} \, \\ \\ \, \frac{\partial \Theta_{[ 0 \, 0]}}{\partial\tau_{11}} &  \frac{\partial \Theta_{[ 0 \, 1]}}{\partial\tau_{11}} &  \frac{\partial \Theta_{[ 1 \, 0]}}{\partial\tau_{11}} & \frac{\partial \Theta_{[ 1 \, 1]}}{\partial\tau_{11}} \, \\ \\  \, \frac{\partial \Theta_{[ 0 \, 0]}}{\partial\tau_{12}} &  \frac{\partial \Theta_{[ 0 \, 1]}}{\partial\tau_{12}} &  \frac{\partial \Theta_{[ 1 \, 0]}}{\partial\tau_{12}} & \frac{\partial \Theta_{[ 1 \, 1]}}{\partial\tau_{12}} \, \\ \\   \, \frac{\partial \Theta_{[ 0 \, 0]}}{\partial\tau_{22}} &  \frac{\partial \Theta_{[ 0 \, 1]}}{\partial\tau_{22}} &  \frac{\partial \Theta_{[ 1 \, 0]}}{\partial\tau_{22}} & \frac{\partial \Theta_{[ 1 \, 1]}}{\partial\tau_{22}} \,   \end{pmatrix}
\end{equation} where $\mu \in \mathds{C}^*$ is a suitable non-zero constant. The square:
\begin{equation}
\label{chi10}
\chi_{10} \coloneqq \chi^2_5 = \prod_{m \in C_{e}^{(2)}} \theta_m^2
\end{equation} is, though, a modular form of weight $10$ with respect to $\Gamma_2$. As concerns the product of all the second order Theta constants:
\begin{equation}
\label{fg}
P_g \coloneqq \prod_{m' \in \mathds{Z_2}^g} \Theta_{m'}
\end{equation}
this one is seen to be a modular form of weight $2^{g-1}$ with respect to $\Gamma_g(2)$ whenever $g \geq 2$.\\
\\

The ring $A(\Gamma_1)$ of modular forms with respect to $\Gamma_1$ is classically known to be generated as a $\mathds{C}$-algebra by the Eisenstein series $E^{(1)}_4$ and $E^{(1)}_6$ respectively of weight $4$ and $6$. Regarding the $g=2$ case, Igusa's structure theorem provides a set of generators (cf. \cite{Igusa7} and \cite{IgusaP}):\\
\begin{teo}
\label{structureigusa}
$A(\Gamma_2) = \mathds{C}[E^{(2)}_4, E^{(2)}_6, \chi_{10}, \chi_{12}, \chi_{35}]$\\
where $E^{(2)}_4$ and $E^{(2)}_6$ are the Eisenstein series respectively of weight $4$ and $6$, $\chi_{10}$ is the modular form described in (\ref{chi10}), $\chi_{12}$ is a modular form of weight $12$ obtained by a suitable symmetrization:
\begin{equation*}
\chi_{12}= \frac{1}{2^{17}3}\sum_{\substack{\{m_{i_1}, \dots , m_{i_6}\} \, \text{s. t.} \\ C_e^{(2)} - \{m_{i_1}, \dots , m_{i_6}\} \in C_4^+}} \pm (\theta_{m_{i_1}} \cdots \theta_{m_{i_6}})^{4}
\end{equation*}
\ni and $\chi_{35} \coloneqq \chi_5 \cdot (azy)_5$, where $\chi_5$ is as in (\ref{chi5}) and $(azy)_5$ is defined by:
\begin{equation*}
(azy)_5 \coloneqq \frac{1}{8} \sum_{\{m_i, m_j, m_k\} \in C_3^-} \pm (\theta_{m_i} \theta_{m_j} \theta_{m_k})^{20}
\end{equation*} where the signs are to be properly chosen in order to to gain the correct symmetrization.
\end{teo}

\ni Due to the isomorphism $\psi_P$ described in (\ref{isomorphismtoS6}), the group $\Gamma_2/\Gamma_2(2)$ admits a sole non trivial irreducible representation of degree $1$, the corresponding character being as follows: 
\begin{equation}
\label{chiP}
\chi_P (\gamma) \coloneqq \left \{ \begin{array}{l} 1 \quad \text{if} \quad \psi_P([\gamma]) \quad \text{is an even permutation} \\ \\ -1 \quad \text{if} \quad \psi_P([\gamma]) \quad \text{is an odd permutation} \end{array} \right.
\end{equation} \\

\ni Then, by setting $\Gamma_2^+ = Ker\chi_P$, the following structure theorem holds (cf. \cite{Igusa7}):

\begin{teo}
\label{structure30}
$A(\Gamma_2^+) = \mathds{C}[E^{(2)}_4, E^{(2)}_6, \chi_{5}, \chi_{12}, (azy)_5]$.
\end{teo}

\begin{note}
\label{note1}
\emph{Amid the generators of $A(\Gamma_2^+)$ described in Theorem \ref{structure30}, $\chi_5$ and $(azy)_5$ are the only ones that transform with the non-trivial character $\chi_P$ under the action of the full modular group $\Gamma_2$. The function $(azy)_5$ is, in particular, the unique modular form of weight $30$ with respect to $\Gamma_2^+$ admitting a non trivial character under the action of $\Gamma_2$.}
\end{note}
\ni The next section is devoted to provide a different expression for $(azy)_5$ by means of a remarkable geometrical construction already described in \cite{Structure}.

\medskip

\medskip

\section{Geometric description of the modular form $(azy)_5$} 

Second order Theta constants are related to Theta constants by Riemann's addition formula. When $g=2$, the following relations hold in particular between the ten non-trivial Theta constants and the four second order Theta constants: 

\begin{equation*}
\begin{array}{lll}
\theta^2_{\left [\substack{0 \, 0 \\ 0 \, 0} \right ]} = \Theta_{[ 0 \, 0]}^2 + \Theta_{[ 0 \, 1]}^2 +\Theta_{[ 1 \, 0]}^2 +\Theta_{[ 1 \, 1]}^2 &\quad \quad & {\theta}^2_{\left [\substack{0 \, 0 \\ 0 \, 1} \right ]} = \Theta_{[ 0 \, 0]}^2 - \Theta_{[ 0 \, 1]}^2 +\Theta_{[ 1 \, 0]}^2 -\Theta_{[ 1 \, 1]}^2\\
\\
{\theta}^2_{\left [\substack{0 \, 0 \\ 1 \, 0} \right ]} = \Theta_{[ 0 \, 0]}^2 +\Theta_{[ 0 \, 1]}^2 - \Theta_{[ 1 \, 0]}^2 - \Theta_{[ 1 \, 1]}^2 &\quad \quad & {\theta}^2_{\left [\substack{0 \, 0 \\ 1 \, 1} \right ]} = \Theta_{[ 0 \, 0]}^2 -\Theta_{[ 0 \, 1]}^2 - \Theta_{[ 1 \, 0]}^2 +\Theta_{[ 1 \, 1]}^2\\
\\
{\theta}^2_{\left [\substack{0 \, 1 \\ 0 \, 0} \right ]} = 2 \Theta_{[ 0 \, 0]} \Theta_{[ 0 \, 1]}  + 2 \Theta_{[ 1 \, 0]}\Theta_{[ 1 \, 1]} &\quad \quad & {\theta}^2_{\left [\substack{1 \, 0 \\ 0 \, 0} \right ]} = 2\Theta_{[0 \, 0]} \Theta_{[ 1 \, 0]} + 2\Theta_{[ 0 \, 1]}\Theta_{[ 1 \, 1]}\\
\\
{\theta}^2_{\left [\substack{1 \, 1 \\ 0 \, 0} \right ]} = 2 \Theta_{[ 0 \, 0]} \Theta_{[ 1 \, 1]}  + 2 \Theta_{[ 0 \, 1]} \Theta_{[ 1 \, 0]} &\quad \quad & {\theta}^2_{\left [\substack{0 \, 1 \\ 1 \, 0} \right ]} = 2\Theta_{[ 0 \, 0]}\Theta_{[ 0 \, 1]} - 2\Theta_{[ 1 \, 0]}\Theta_{[ 1 \, 1]}\\
\\
{\theta}^2_{\left [\substack{1 \, 0 \\ 0 \, 1} \right ]} = 2 \Theta_{[ 0 \, 0]} \Theta_{[ 1 \, 0]}  - 2 \Theta_{[ 0 \, 1]} \Theta_{[ 1 \, 1]} &\quad \quad & {\theta}^2_{\left [\substack{1 \, 1 \\ 1 \, 1} \right ]} = 2\Theta_{[ 0 \, 0]} \Theta_{[ 1 \, 1]} - 2\Theta_{[ 0 \, 1]} \Theta_{[ 1 \, 0]}
\end{array}
\end{equation*}\\

\ni A quadratic form $Q_m$ in the variables $X_1, X_2, X_3, X_4$ is, therefore, associated to each even $2$-characteristic $m$:
\begin{equation*}
m \longmapsto \, Q_m \quad \quad \quad \quad \text{where} \quad \theta_m^2=Q_m(\Theta_1, \Theta_2, \Theta_3, \Theta_4)
\end{equation*}
\ni Hence, a quadric $V_m$ in the projective space $\mathds{P}^3$ also corresponds to each even $2$-characteristic $m$:
\begin{equation*}
m \longmapsto \, V_m \coloneqq V(Q_m) = \{\, [X_1, X_2, X_3, X_4] \in  \mathds{P}^3 \, \mid \, \, \, Q_m(X_1, X_2, X_3, X_4) = 0 \, \}
\end{equation*} 

\medskip

\ni Furthermore, for each $4$-plet $M \in C_4^+$ the set
\begin{equation*}
\bigcap_{m \in M^c} V_m \subset \mathds{P}^3
\end{equation*} 
where $M^c$ stands for the $6$-plet of even $2$-characteristics being complementary in $C_e^{(2)}$ to $M$, contains exactly four points (cf. \cite{Structure}). A configuration of four hyperplanes in $\mathds{P}^3$ is thus uniquely determined for any $M \in C_4^+$, each hyperplane being characterized by passing through all except one of the four points contained in the set $\bigcap_{m \in M^c} V_m$. Therefore, a collection of four linear forms describing these four hyperplanes is found to be associated to each $M \in C_4^+$:\\

\begin{equation*}
\left \{ \begin{array}{l} \psi^M_1 \coloneqq \psi^M_1 (X_1, X_2, X_3, X_4) \\
\\
\psi^M_2 \coloneqq \psi^M_2 (X_1, X_2, X_3, X_4)\\ 
\\
\psi^M_3 \coloneqq \psi^M_3 (X_1, X_2, X_3, X_4)\\ 
\\
\psi^M_1 \coloneqq \psi^M_1 (X_1, X_2, X_3, X_4) 
\end{array} \right. 
\end{equation*}\\

\ni Hence, a tetrahedron $T_M$ in the projective space $\mathds{P}^3$ is uniquely determined by each $4$-plet $M=(m_1, m_2, m_3, m_4) \in C_4^+$, the set $\bigcap_{m \in M^c} V_m$ being in fact the set of its vertices.\\

\ni A remarkable holomorphic function can be associated to each $M \in C_4^+$ by means of the  functions $\psi^M_i$:\\

\begin{equation*}
\begin{array}{l}
F_M(\tau) \coloneqq \tilde{F}_M (\, \Theta_{[ 0 \, 0]}(\tau), \Theta_{[ 0 \, 1]}(\tau), \Theta_{[ 1 \, 0]}(\tau),\Theta_{[ 1 \, 1]}(\tau)\,) \\
\\
\text{where} \quad \tilde{F}_M \coloneqq \prod_{i=1}^4 \psi^M_i
\end{array}
\end{equation*}\\

\ni In particular, the $4$-plet
\begin{equation*}
M_0 \coloneqq \left \{ \begin{bmatrix} 0 0 \\ 0 0 \end{bmatrix},  \begin{bmatrix} 0 0 \\ 0 1 \end{bmatrix}, \begin{bmatrix} 0 0 \\ 1 0 \end{bmatrix}, \begin{bmatrix} 0 0 \\ 1 1 \end{bmatrix} \right \}  \in C_4^+
\end{equation*} is such that: 
\begin{equation*}
\bigcap_{m \in M_0^c} V_m = \{[1,0,0,0], [0,1,0,0], [0,0,1,0], [0,0,0,1] \}
\end{equation*}
\ni the faces of the corresponding tetrahedron $T_{M_0}$ being thus described by:
\begin{equation*}
\begin{array}{cccc}
\psi^{M_0}_1 = X_1; & \psi^{M_0}_2 = X_2; & \psi^{M_0}_1 = X_3; & \psi^{M_0}_4 = X_4; 
\end{array}
\end{equation*} \\

\ni Hence:
\begin{equation*}
F_{M_0}(\tau) = \prod_{m' \in \mathds{Z_2}^2} \Theta_{m'}=P_2(\tau)
\end{equation*}
and $F_{M_0}$ is, therefore, a modular form of weight $2$ with respect to $\Gamma_2(2)$ (cf. (\ref{fg})\,).\\

\medskip

\begin{lem}
\label{st}
The group $\Gamma_{2,0}(2)$ is the stabilizer $St_{M_0}$ of the $4$-plet $M_0$. 
\end{lem}

\begin{proof}
On the one hand $\Gamma_{2,0}(2) \subset St_{M_0}$ by definition; on the other hand, $\gamma \in St_{M_0}$ implies $diag({}^tcd) - cm'' \equiv 0 \, \text{mod}\, 2$ for each $m'' \in \mathds{Z}^2_2$, hence $\gamma \in \Gamma_{2,0}(2)$.
\end{proof}

\ni As $C_4^+$ is an orbit, Lemma \ref{st} self-evidently implies the following:
\begin{coro}
$[\Gamma_2 : \Gamma_{2,0}(2)]=15$
\end{coro}

\ni The map $[\gamma] \mapsto T_{\gamma M_1}$ is thus a bijection between $\Gamma_2 / \Gamma_{2,0}(2)$ and the collection $\{T_M\}_{M \in C_4^+}$ of the tetrahedrons; the product of all the images of $F_{M_0}$ under the action of the representatives of the fifteen cosets of $\Gamma_{2,0}(2)$ in $\Gamma_2$ is then a proper candidate to be focused on. However, the behaviour of $F_{M_0}$ under the action of $\Gamma_{2,0}(2)$ is due to be investigated foremost.
\begin{prop}
\label{f0transform}
Let $\chi_P$ be the character introduced in (\ref{chiP}); then, with reference to the action in (\ref{azf1}), one has:
\begin{equation*}
\eta^{-1} |_2 F_{M_0}= \chi_ P(\eta) \, F_{M_0} \quad \quad \quad \quad \forall \eta \in \Gamma_{2,0}(2)
\end{equation*}

\end{prop}

\begin{proof}
One only needs to check that $\chi_P$ is the very character involved in the transformation law for $F_{M_0}$ under the action of $\Gamma_{2,0}(2)$ . Since one has:
\begin{equation*}
F_{M_0}(\eta_0 \tau) = - F_{M_0} (\tau) \quad \quad \quad \quad \text{for} \quad \eta_0=\begin{pmatrix} 1 & 0 & 1 & 0 \\ 0 & 1 & 0 & 0 \\ 0 & 0 & 1 & 0 \\ 0 & 0 & 0 & 1 \end{pmatrix} \in \Gamma_{2,0}(2)
\end{equation*}
then $F_{M_0}$ transforms with a non-trivial character $\chi$ under the action of $\Gamma_{2,0}(2)$; the character $\chi$ is of course trivial on $\Gamma_{2,0}(2)^+ \coloneqq Ker\chi \supset \Gamma_2(2)$, thus extending to a non-trivial character of $\Gamma_2$ (cf. \cite{Ibukiyama} Appendix), which is well defined on $\Gamma_2/\Gamma_2(2)$; hence $\chi$ extends to $\chi_P$, this one being the sole non-trivial character of $\Gamma_2/\Gamma_2(2)$.
\end{proof}

\medskip

\ni A holomorphic function is well defined for any fixed $\gamma \in \Gamma_2$:
\begin{equation*}
\varphi_{\gamma}(\tau) \coloneqq \chi_P(\gamma)^{-1}(\gamma^{-1}|_2 F_{M_0}) (\tau)
\end{equation*}
\ni In particular, for a fixed choice $\gamma_1, \cdots \gamma_{15}$ of the representatives of the cosets of $\Gamma_{2,0}(2)$ in $\Gamma_2$ a notable holomorphic function is defined:
\begin{equation*}
\varphi(\tau) \coloneqq \prod_{i=1}^{15}\varphi_{\gamma_i}(\tau)
\end{equation*}\\

\begin{prop}
\label{welldefined}
The function $\varphi$ does not depend on the choice of the coset representatives.
\end{prop}

\begin{proof}
If $\gamma_1, \dots ,\gamma_{15}$ and $\gamma'_1, \dots , \gamma'_{15}$ are two different choices for the representatives of the cosets of $\Gamma_{2,0}(2)$ in $\Gamma_2$, then there exists a permutation $j$ of the indices such that for each $i=1, \dots , 15$ one has $\gamma'_{j(i)}=\eta_i\gamma_i$ with $\eta_i \in \Gamma_{2,0}(2)$. Hence, by (\ref{cocycle}) and Proposition \ref{f0transform}:
\begin{equation*}
\begin{split}
\prod_{j=1}^{15}\varphi_{\gamma'_j}(\tau) & = \prod_{i=1}^{15}\varphi_{\gamma'_{j(i)}}(\tau) = \prod_{i=1}^{15}\varphi_{\eta_i\gamma_i}(\tau) = \prod_{i=1}^{15}\chi_P(\eta_i\gamma_i)^{-1}({(\eta_i\gamma_i)}^{-1}|_2 F_{M_0}) (\tau)=\\
 & = \prod_{i=1}^{15}\chi_P(\eta_i\gamma_i)^{-1} \, D(\eta_i \gamma_i , \tau)^{-2} \, \chi_P(\eta_i) \, D(\eta_i , \gamma_i \tau)^2 \, F_{M_0}(\gamma_i\tau)=\\
& =  \prod_{i=1}^{15}\chi_P(\gamma_i)^{-1} \, D(\gamma_i , \tau)^{-2} \, F_{M_0}(\gamma_i\tau) =  \prod_{i=1}^{15}\varphi_{\gamma_i}(\tau)
\end{split}
\end{equation*}
\end{proof}

\begin{teo}
\label{modular}
$\varphi$ is a modular form of weight $30$ with respect to $\Gamma_2$ with character $\chi_P$.
\end{teo}

\begin{proof}
The transformation formula $\varphi(\gamma \tau) =  \chi_P(\gamma) \, D(\gamma , \tau)^{30}\varphi(\tau)$ has to be proved whenever $\gamma \in \Gamma_2$. Let thus $\gamma_1, \dots , \gamma_{15}$ be a fixed collection of coset representatives of $\Gamma_{2,0}(2)$ in $\Gamma_2$; then, for each $\gamma \in \Gamma_2$ one has:
\begin{equation*}
\begin{split}
\varphi(\gamma \tau) & = \prod_{i=1}^{15}\varphi_{\gamma_i}(\gamma \tau) = \prod_{i=1}^{15}\chi_P(\gamma_i)^{-1}D(\gamma_i , \gamma \tau)^{-2} \, F_{M_0}(\gamma_i\gamma\tau)=\\
& = \prod_{i=1}^{15}\chi_P(\gamma_i)^{-1} \, D(\gamma_i \gamma , \tau)^{-2} \, D(\gamma , \tau)^2 \, F_{M_0}(\gamma_i\gamma\tau)=\\
& = D(\gamma , \tau)^{30} \prod_{i=1}^{15}\chi_P(\gamma) \, \varphi_{\gamma_i\gamma}(\tau)=\\
& = \chi_P(\gamma)^{15}  D(\gamma , \tau)^{30} \prod_{i=1}^{15}\varphi_{\gamma_i\gamma}(\tau) = \, \, \chi_P(\gamma) \, D(\gamma , \tau)^{30} \varphi(\tau)
\end{split}
\end{equation*}
\ni the last equality being due to Proposition \ref{welldefined}.
\end{proof}

\medskip

\medskip

\begin{coro}
There exists $\lambda \in \mathds{C}^*$ such that $\varphi = \lambda \, (azy)_5$.
\end{coro}

\begin{proof}
By Theorem \ref{modular} $\varphi$ is a modular form of weight $30$ with respect to $\Gamma_2^+$ transforming with a non-trivial character under the action of $\Gamma_2$; hence, the thesis follows by Theorem \ref{structure30} and Note \ref{note1}.
\end{proof}

\medskip

\ni Such an expression for the modular form $(azy)_5$ has been recently found by Gehre and Krieg in a different way by means of quaternionic Theta constants (cf. \cite{Krieg}).


\medskip


\addcontentsline{toc}{chapter}{Bibliography}

\begin{thebibliography}{9}
\bibitem[GS93]{Structure}{\it B. Van Geemen, D. Van Straten}, The cusp forms of weight $3$ on $\Gamma_2(2,4,8)$, Mathematics of Computation 61 (1993) 
\bibitem[Ib91]{Ibukiyama}{\it T. Ibukiyama}, On Siegel Modular Varieties of Level 3, International Journal of Mathematics, Volume 2 (1991)
\bibitem[Ig64]{Igusa7} {\it J. Igusa}, On Siegel modular forms of genus two (II), American Journal of Mathematics 86 (1964)
\bibitem[Ig66]{Igusa10}{\it J. Igusa}, On the graded ring of theta-constants (II), American Journal of Mathematics 88 (1966)
\bibitem[Ig67]{IgusaP} {\it J. Igusa}, Modular forms and projective invariants, American Journal of Mathematics 89 (1967)
\bibitem[Ig83]{Igusa2} {\it J. Igusa}, Multiplicity one theorem and problems related to Jacobi's formula, American Journal of Mathematics 105 (1983)
\bibitem[GK10]{Krieg}{\it D. Gehre, A. Krieg}, Quathernionic Theta constants, Arch. Math. 94 (2010)
\bibitem[SM94]{SM} {\it R. Salvati Manni}, Modular varieties with level 2 Theta strucutre, American Journal of Mathematics 116 (1994) n.6
\end{thebibliography}
\end{document}